\documentclass{amsart}
\usepackage[utf8x]{inputenc}

\usepackage{amsthm,amsmath,amssymb,amstext,amsfonts}

\newcommand{\field}[1]{\mathbb{#1}}

\newcommand{\N}{\field{N}}

\numberwithin{equation}{section}
\newtheorem{theorem}{Theorem}[section]
\newtheorem{lemma}[theorem]{Lemma}

\newtheorem{proposition}[theorem]{Proposition}
\theoremstyle{remark}

\renewenvironment{proof}[1][Proof]{\begin{trivlist}
\item[\hskip \labelsep {\bfseries #1:}]}{\qed\end{trivlist}}

\title{Siladi\'c's theorem: weighted words, refinement and companion}
\author{Jehanne Dousse}
\date{\today}

\begin{document}

\begin{abstract}
In a previous paper, the author gave a combinatorial proof and refinement of Siladi\'c's theorem, a Rogers-Ramanujan type partition identity arising from the study of Lie algebras. Here we use the basic idea of the method of weighted words introduced by Alladi and Gordon to give a non-dilated version, further refinement and companion of Siladi\'c's theorem. However, while in the work of Alladi and Gordon, identities were proved by doing transformations on generating functions, we use recurrences and $q$-difference equations as the original method seems difficult to apply in our case.
As the non-dilated version features the same infinite product as Schur's theorem, another dilation allows us to find a new interesting companion of Schur's theorem, with difference conditions very different from the original ones.
\end{abstract}

\maketitle

\section{Introduction}
A partition of $n$ is a non-increasing sequence of positive integers whose sum is $n$.
For example, there are $5$ partitions of $4$: $4$, $3+1$, $2+2$, $2+1+1$ and $1+1+1+1$.

An important field in the theory of partitions and $q$-series is the study of partition identities of the Rogers-Ramanujan type.
The Rogers-Ramanujan identities~\cite{RogersRamanujan}, first discovered by Rogers in 1894 and rediscovered by Ramanujan in 1917 are the following $q$-series identities:
\begin{theorem}
\label{rr}
Let $a=0$ or $1$. Then
\begin{equation*}
\sum_{n=0}^{\infty} \frac{q^{n(n+a)}}{(1-q)(1-q^2)\cdots (1-q^n)} = \prod_{k=0}^{\infty} \frac{1}{(1-q^{5k+a+1})(1-q^{5k+4-a})}.
\end{equation*}
\end{theorem}
These analytic identities can be interpreted in terms of partitions in the following way:
\begin{theorem}
Let $a=0$ or $1$. Then for every natural number $n$, the number of partitions of $n$ such that the difference between two consecutive parts is at least $2$ and the part $1$ appears at most $1-a$ times is equal to the number of partitions of $n$ into parts congruent to $\pm (1+a) \mod 5.$
\end{theorem}
More generally, a theorem of the type ``for all $n$, the number of partitions of $n$ with some difference conditions equals the number of partitions of $n$ with some congruence conditions" is called a partition identity of the Rogers-Ramanujan type.

The method of weighted words was introduced by Alladi and Gordon~\cite{Alladi,AllGor} to give a refinement of another famous Rogers-Ramanujan type partition identity : Schur's theorem~\cite{Schur}.

\begin{theorem}[Schur]
\label{schur}
For any integer $n$, let $A(n)$ denote the number of partitions of $n$ into distinct parts congruent to $1$ or $2$ modulo $3$ and $B(n)$ the number of partitions of $n$ such that parts differ by at least $3$ and no two consecutive multiples of $3$ appear. Then for all $n$,
$$A(n)=B(n).$$
\end{theorem}

The idea of the method is to consider integers appearing in different colours $a$,$b$,$c$,... and to view a partition with difference conditions as a sequence of coloured integers with an order on the colours and difference conditions on the numbers. The colours also represent free parameters which allow to refine the theorems. The method of weighted words is usually used at the non-dilated level, which means that one obtains the final theorem by doing dilations of the type
$$q \rightarrow q^M, a \rightarrow a q^{-m_a}, b \rightarrow bq^{-m_b}, \dots ,$$
where $M, m_a, m_b$ are some positive integers.

For example, in the case of Schur's theorem, Alladi and Gordon considered integers in three colours
\begin{equation} \label{colororder}
ab < a < b,
\end{equation}
such that
$$1_{ab} < 1_a < 1_b < 2_{ab} < 2_a < 2_b < \cdots,$$
and the dilations
\begin{equation}
\label{dilschur}
q \rightarrow q^3, a \rightarrow a q^{-2}, b \rightarrow bq^{-1}.
\end{equation}
Then, denoting by $c(\lambda)$ the colour of $\lambda$, the non-dilated refinement of Schur's theorem is
\begin{theorem}[Alladi-Gordon]
\label{schurnondil}
Let $A(u,v,n)$ be the number of partitions of $n$ into $u$ distinct parts coloured $a$ and $v$ distinct parts coloured $b$.
Let $B(u,v,n)$ be the number of partitions $\lambda_1 + \cdots + \lambda_s$ of $n$ into distinct parts with no part $1_{ab}$, such that the difference $\lambda_i - \lambda_{i+1} \geq 2$ if $c(\lambda_{i}) = ab$ or $c(\lambda_i) < c(\lambda_{i+1})$ in \eqref{colororder}, having $u$ parts $a$ or $ab$ and $v$ parts $b$ or $ab$.

Then for all $u,v,n \in \N,$ $A(u,v,n) = B(u,v,n).$
\end{theorem}
To prove this theorem, Alladi and Gordon used the fact that a partition with $n$ parts satisfying the difference conditions is a minimal partition with $n$ parts satisfying the difference conditions to which one has added a partition having at most $n$ parts. So they computed the generating function for such minimal partitions using $q$-binomial coefficients and concluded using a $q$-series identity to transform the generating function of partitions with difference conditions into the generating function for partitions with congruence conditions.

Then the method was used again by Alladi, Andrews and Gordon~\cite{AllAndGor2} to refine and generalise a famous theorem of G\"ollnitz~\cite{Gollnitz}, and yet again by the same authors~\cite{AllAndGor} to give a refinement of Capparelli's theorem~\cite{Capparelli}, a partition identity which arose in the study of Lie algebras.

In this paper, we shall give the same type of refinement/non-dilated version for Siladi\'c's identity~\cite{Siladic}, a Rogers-Ramanujan type identity which was proved by studying representations of the twisted affine Lie algebra $A^{(2)}_2.$

\begin{theorem}[Siladi\'c]
\label{siladic}
The number of partitions $\lambda_1 +\cdots+ \lambda_s$ of an integer $n$ into parts different from $2$ such that difference between two consecutive parts is at least $5$ (ie. $\lambda_i - \lambda_{i+1} \geq 5$) and
\begin{equation*}
\begin{aligned}
&\lambda_i - \lambda_{i+1} = 5 \Rightarrow \lambda_i + \lambda_{i+1} \not \equiv \pm 1, \pm 5, \pm 7 \mod 16,
\\&\lambda_i - \lambda_{i+1} = 6 \Rightarrow \lambda_i + \lambda_{i+1} \not \equiv \pm 2, \pm 6 \mod 16,
\\&\lambda_i - \lambda_{i+1} = 7 \Rightarrow \lambda_i + \lambda_{i+1} \not \equiv \pm 3 \mod 16,
\\&\lambda_i - \lambda_{i+1} = 8 \Rightarrow \lambda_i + \lambda_{i+1} \not \equiv \pm 4 \mod 16,
\end{aligned}
\end{equation*}
is equal to the number of partitions of $n$ into distinct odd parts.
\end{theorem}

In a previous paper~\cite{Doussesil}, we had already given a reformulation, combinatorial proof and (less precise) refinement of Siladi\'c's theorem, without using the method of weighted words.

\begin{theorem}[Dousse]
\label{refinement}
Let $n \in \N$ and $k \in \N^*$. Let $A(k,n)$ denote the number of partitions of $n$ into $k$ distinct odd parts, and $B(k,n)$ denote the number of partitions $\lambda_1 + \cdots+ \lambda_s$ of $n$ such that $k$ equals the number of odd part plus twice the number of even parts, satisfying the following conditions:
\begin{enumerate}
  \item $\forall i \geq 1, \lambda_i \neq 2$,
  \item $\forall i \geq 1, \lambda_i - \lambda_{i+1} \geq 5$,
  \item $\forall i \geq 1$,
  \begin{equation*}
  \begin{aligned}
&\lambda_i - \lambda_{i+1} = 5 \Rightarrow \lambda_i \equiv 1, 4 \mod 8,
\\&\lambda_i - \lambda_{i+1} = 6 \Rightarrow \lambda_i \equiv 1, 3, 5, 7 \mod 8,
\\&\lambda_i - \lambda_{i+1} = 7 \Rightarrow \lambda_i \equiv 0, 1, 3, 4, 6, 7 \mod 8,
\\&\lambda_i - \lambda_{i+1} = 8 \Rightarrow \lambda_i \equiv 0, 1, 3, 4, 5, 7 \mod 8.
	\end{aligned}
	\end{equation*}
\end{enumerate}
Then $A(k,n)=B(k,n)$.
\end{theorem}

The method of weighted words will allow us to obtain not only a refinement where we keep track of the number of odd parts, but a non-dilated version which can lead to an infinitude of new identities and a refinement where we distinguish which parts are congruent to $1$ or $3$ modulo $4$. We now describe this non-dilated version.

Let us consider the integers appearing in five colours $a$, $b$, $ab$, $a^2$ and $b^2$, ordered as follows:
$$1_{ab} < 1_a < 1_{b^2} <1_{b} <2_{ab} < 2_a <3_{a^2} < 2_{b} <3_{ab} < 3_a < 3_{b^2} <3_b < \cdots .$$
Note that the colours $a^2$ and $b^2$ only appear for odd integers. It is the first time that the method of weighted words is used with ``squared'' colours.

We consider partitions $\lambda_1 + \cdots + \lambda_s$ where the entry $(x,y)$ in the matrix $A$ gives the minimal difference between $\lambda_i$ of colour $x$ and $\lambda_{i+1}$ of colour $y$:
$$
A=\bordermatrix{\text{} & a & b & ab & a^2 & b^2 \cr a_{odd} & 2&2&1&2&2 \cr b^2 &2&3&2&2&4 \cr b_{odd} &1&2&1&2&2 \cr ab_{even} &2&2&2&3&3 \cr a_{even}&2&2&2&3&3 \cr a^2 &3&3&3&4&4 \cr b_{even} &1&2&1&1&3 \cr ab_{odd} &2&3&2&2&3}.
$$
This definition might seem complicated or unnatural, but under the dilations
\begin{equation}
\label{dilat}
\begin{aligned}
q &\rightarrow q^4,\\
a &\rightarrow aq^{-3},\\
b &\rightarrow bq^{-1},
\end{aligned}
\end{equation}
the order on the coloured integers becomes the natural ordering
$$0_{ab} < 1_a < 2_{b^2} <3_{b} <4_{ab} < 5_a <6_{a^2} < 7_{b} <8_{ab} < 9_a < 10_{b^2} <11_b < \cdots, $$
and the difference conditions become those of Siladi\'c's theorem.

Therefore the non-dilated version of Siladi\'c's theorem is the following.
\begin{theorem}
\label{th:nondil}
For $u,v,n \in \N$, let $D(u,v,n)$ denote the number of partitions $\lambda_1 + \cdots + \lambda_s$ of $n$, with no part $1_{ab}$ or $1_{b^2}$, satisfying the difference conditions given by the matrix $A$, such that $u$ equals the number of parts $a$ or $ab$ plus twice the number of parts $a^2$ and $v$ equals the number of parts $b$ or $ab$ plus twice the number of parts $b^2$.

Then
\begin{equation*}
\sum_{u,v,n \in \N} D(u,v,n)a^ub^vq^n = \prod_{n\geq 1} \left(1+aq^n\right)\left(1+bq^n\right).
\end{equation*}
\end{theorem}

In the original proofs of Alladi, Andrews and Gordon using the method of weighted words, they considered minimal partitions satisfying the difference conditions to obtain the generating functions and concluded by a $q$-series manipulation. In this paper we proceed completely differently as the difference conditions are much more complicated and it is not clear how to obtain the minimal partitions and generating functions, let alone find a $q$-series transformation which would lead to the correct infinite product. The proof of Theorem~\ref{th:nondil} is similar to the one of Theorem~\ref{refinement} in~\cite{Doussesil} and uses $q$-difference equations and recurrences, but here we work at the non-dilated level and keep track of the letters $a$ and $b$ at each step. This shows that $q$-difference equations can be a good approach to finding refinements when the difference conditions are intricate.

By doing the dilations~\eqref{dilat}, we obtain the following new refinement of Siladi\'c's theorem, more precise than Theorem~\ref{refinement}.

\begin{theorem}
\label{th:refdilat}
For $u,v,n \in \N$, let $C_4(u,v,n)$ denote the number of partitions of $n$ into $u$ distinct parts congruent to $1$ modulo $4$ and $v$ distinct parts congruent to $3$ modulo $4$.
Let $D_4(u,v,n)$ denote the number of partitions $\lambda_1 + \cdots + \lambda_s$ of $n$ such that $u$ equals the number of parts congruent to $0$ or $1$ modulo $4$ plus twice the number of parts congruent to $6$ modulo $8$ and $v$ equals the number of parts congruent to $0$ or $3$ modulo $4$ plus twice the number of parts congruent to $2$ modulo $8$, satisfying the following conditions:
\begin{enumerate}
  \item $\forall i \geq 1, \lambda_i \neq 2$,
  \item $\forall i \geq 1, \lambda_i - \lambda_{i+1} \geq 5$,
  \item $\forall i \geq 1$,
  \begin{equation*}
  \begin{aligned}
&\lambda_i - \lambda_{i+1} = 5 \Rightarrow \lambda_i \equiv 1, 4 \mod 8,
\\&\lambda_i - \lambda_{i+1} = 6 \Rightarrow \lambda_i \equiv 1, 3, 5, 7 \mod 8,
\\&\lambda_i - \lambda_{i+1} = 7 \Rightarrow \lambda_i \equiv 0, 1, 3, 4, 6, 7 \mod 8,
\\&\lambda_i - \lambda_{i+1} = 8 \Rightarrow \lambda_i \equiv 0, 1, 3, 4, 5, 7 \mod 8.
	\end{aligned}
	\end{equation*}
\end{enumerate}
Then $C_4(u,v,n)=D_4(u,v,n)$.
\end{theorem}

Moreover, the non-dilated Theorem~\ref{th:nondil} allows to obtain an infinitude of new identities by doing different dilations. For example, the dilation
\begin{equation}
\label{dil2}
\begin{aligned}
q &\rightarrow q^4,\\
a &\rightarrow aq^{-1},\\
b &\rightarrow bq^{-3},
\end{aligned}
\end{equation}
give the following interesting companion of Siladi\'c's theorem.

\begin{theorem}
\label{th:comp}
For $u,v,n \in \N$, let $C'_4(u,v,n)$ denote the number of partitions of $n$ into $u$ distinct parts congruent to $3$ modulo $4$ and $v$ distinct parts congruent to $1$ modulo $4$.
Let $D'_4(u,v,n)$ denote the number of partitions $\lambda_1 + \cdots + \lambda_s$ of $n$ such that $u$ equals the number of parts congruent to $0$ or $3$ modulo $4$ plus twice the number of parts congruent to $2$ modulo $8$ and $v$ equals the number of parts congruent to $0$ or $1$ modulo $4$ plus twice the number of parts congruent to $6$ modulo $8$, satisfying the following conditions:
\begin{enumerate}
  \item $\forall i \geq 1, \lambda_i \neq 2$,
  \item $$\lambda_i - \lambda_{i+1} 
\begin{cases}
=5,6,8,9 \text{ or } \geq 11 \text{ if } \lambda_{i} \equiv 0 \mod 8, \\
=2 \text{ or } \geq 5 \text{ if } \lambda_{i} \equiv 1 \mod 8, \\
=11 \text{ or } \geq 13 \text{ if } \lambda_{i} \equiv 2 \mod 8, \\
\geq 7 \text{ if } \lambda_{} \equiv 3 \mod 8, \\
=5 \text{ or } \geq 7 \text{ if } \lambda_{i} \equiv 4 \mod 8, \\
=2,3,5,6 \text{ or } \geq 8 \text{ if } \lambda_{i} \equiv 5 \mod 8, \\
=3,4,6,7 \text{ or } \geq 9 \text{ if } \lambda_{i} \equiv 6 \mod 8, \\
=8 \text{ or } \geq 10 \text{ if } \lambda_{i} \equiv 7 \mod 8.
\end{cases}$$
\end{enumerate}

Then $C'_4(u,v,n)=D'_4(u,v,n)$.
\end{theorem}

Note that the infinite product in Theorem~\ref{th:nondil} is exactly the same as in Alladi and Gordon's non dilated version of Schur's theorem (Theorem~\ref{schurnondil}). Thus doing the dilation~\eqref{dilschur} in Theorem~\ref{th:nondil} gives exactly the infinite product of Schur's theorem $$\prod_{k=0}^{\infty} (1+aq^{3k+1})(1+bq^{3k+2}),$$ but very different weighted words and difference conditions. Indeed, we have some numbers which appear with``squared'' colours, such as $5_{a^2}$, which Alladi and Gordon never considered and their papers on the method of weighted words.

With the dilation~\eqref{dilschur}, the ordering of integers 
$$1_{ab} < 1_a < 1_{b^2} <1_{b} <2_{ab} < 2_a <3_{a^2} < 2_{b} <3_{ab} < 3_a < 3_{b^2} <3_b < \cdots $$
becomes
$$0_{ab} < 1_a < 1_{b^2} <2_{b} <3_{ab} < 4_a <5_{a^2} < 5_{b} <6_{ab} < 7_a < 7_{b^2} <8_b < \cdots .$$
So the integers congruent to $\pm 1 \mod 6$ can appear in squared colours.
To state the theorem, we will use overpartitions (partitions in which the first occurence of a part may be overlined) and consider that the parts with squared colours are overlined parts. We obtain the following companion of Schur's theorem.

\begin{theorem}
\label{newschur}
For $u,v,n \in \N$, let $C_3(u,v,n)$ denote the number of partitions of $n$ into $u$ distinct parts congruent to $1$ modulo $3$ and $v$ distinct parts congruent to $2$ modulo $3$.
Let $D_3(u,v,n)$ denote the number of overpartitions $\lambda_1 + \cdots + \lambda_s$ of $n$ such that only parts congruent to $\pm 1 \mod 6$ can be overlined, $\overline{1}$ is not a part, and such that $u$ equals the number of parts congruent to $0$ or $1$ modulo $3$ plus twice the number of overlined parts congruent to $5$ modulo $6$ and $v$ equals the number of parts congruent to $0$ or $2$ modulo $3$ plus twice the number of overlined parts congruent to $1$ modulo $6$, satisfying the following conditions:
$$\lambda_i - \lambda_{i+1}  \geq
\begin{cases}
4 + \chi(\overline{\lambda_{i+1}}) \text{ if } \lambda_{i} \equiv 1, 2, 3, 5 \mod 6 \text{ and is not overlined} , \\
5 + \chi(\overline{\lambda_{i+1}}) \text{ if } \lambda_{i} \equiv 0, 4 \mod 6, \\
6 + \chi(\overline{\lambda_{i+1}}) \text{ if } \lambda_{i} \equiv 1, 5 \mod 6 \text{ and is overlined},
\end{cases}$$
where $$\chi(\overline{\lambda_{i+1}}) = 
\begin{cases}
= 1 \text{ if $\lambda_{i+1}$ is overlined} , \\
= 0 \text{ otherwise}.
\end{cases}$$

Then $C_4(u,v,n)=D_4(u,v,n)$.
\end{theorem}

Let us now describe the proof of Theorem~\ref{th:nondil}.

\section{Idea of the proof}
To prove Theorem~\ref{th:nondil}, we proceed as follows.

For $n,k,u,v \in \N$, let $d_k(u,v,n)$ (resp. $e_k(u,v,n)$) denote the number of partitions $\lambda_1+\cdots+\lambda_s$ counted by $D(u,v,n)$ such that the largest part $\lambda_1$ is at most (resp. is equal to) $k$.
We define, for $|a|, |b|, |q|<1$, $k \in \N^*$, 
\begin{equation*}
G_k (a,b,q) = 1+ \sum_{u=0}^{\infty} \sum_{v=0}^{\infty} \sum_{n=1}^{\infty} d_k(u,v,n) a^u b^v q^n.
\end{equation*}
Thus $G_{\infty}(a,b,q):= \lim_{k \rightarrow \infty} G_k (a,b,q)$ is the generating function for the partitions counted by $D(u,v,n)$, as there is no more restriction on the size of the largest part.

We start by giving recurrences for $d_k(u,v,n)$ by using the difference conditions of matrix $A$ and combinatorial reasoning on the largest part of the partitions. Then we use them to obtain simple $q$-difference equations for $G_k (a,b,q)$. This is done in Section~\ref{qdiff}.

Our goal is to prove the following key proposition. This is done by induction in Section~\ref{induction}, using the above-mentioned $q$-difference equations.

\begin{proposition}
\label{keyprop}
For all $k \in \N^*,$
\begin{align}
G_{2k+1_{ab}}(a,b,q) &= (1+aq) G_{2k_{a}}(b,aq,q), \label{key1}\\
G_{2k+1_{b^2}}(a,b,q) &= (1+aq) G_{2k_{b}}(b,aq,q), \label{key2}\\
G_{2k+2_{ab}}(a,b,q) &= (1+aq) G_{2k+1_{a}}(b,aq,q), \label{key3}\\
G_{2k+1_{a^2}}(a,b,q) &= (1+aq) G_{2k-1_{b}}(b,aq,q). \label{key4}
\end{align}
\end{proposition}

Indeed we can then let $k$ go to infinity and deduce 

\begin{align*}
G_{\infty}(a,b,q) &= (1+aq) G_{\infty} (b,aq,q)\\
&= (1+aq)(1+bq) G_{\infty} (aq,bq,q)\\
&= (1+aq)(1+bq)(1+aq^2)(1+bq^2) G_{\infty} (aq^2,bq^2,q)\\
&= \cdots \\
&= \prod_{k=1}^{\infty} \left(1+ aq^{k}\right)(1+bq^k) \times G_{\infty}(0,0,q)\\
&= \prod_{k=1}^{\infty} \left(1+ aq^{k}\right)(1+bq^k).
\end{align*}
This is the generating function for partitions into distinct parts coloured $a$ or $b$, and it proves Theorem~\ref{th:nondil}.

\section{Recurrences and $q$-difference equations}
\label{qdiff}
We now use combinatorial reasoning on the largest part of partitions to state some recurrences. We have the following identities:

\begin{lemma}
\label{eqd}
For all $u,v \in \N$, $k,n \in \N^*,$ we have
\begin{equation}
\label{eqd1}
d_{2k+1_{ab}}(u,v,n)=d_{2k_{b}}(u,v,n)+d_{2k-1_{a}}(u-1,v-1,n-2k-1),
\end{equation}
\begin{equation}
\label{eqd2}
d_{2k+1_{a}}(u,v,n)=d_{2k+1_{ab}}(u,v,n)+d_{2k_{ab}}(u-1,v,n-2k-1),
\end{equation}
\begin{equation}
\label{eqd3}
d_{2k+1_{b^2}}(u,v,n)=d_{2k+1_{a}}(u,v,n)+d_{2k-1_{a}}(u,v-2,n-2k-1),
\end{equation}
\begin{equation}
\label{eqd4}
d_{2k+1_{b}}(u,v,n)=d_{2k+1_{b^2}}(u,v,n)+d_{2k_{a}}(u,v-1,n-2k-1),
\end{equation}
\begin{equation}
\label{eqd5}
\begin{aligned}
d_{2k+2_{ab}}(u,v,n)= &\ d_{2k+1_{b}}(u,v,n)+d_{2k_{a}}(u-1,v-1,n-2k-2)\\
&+d_{2k-1_{a}}(u-1,v-2,n-4k-2),
\end{aligned}
\end{equation}
\begin{equation}
\label{eqd6}
\begin{aligned}
d_{2k+2_{a}}(u,v,n)=& \ d_{2k+2_{ab}}(u,v,n)+d_{2k_{a}}(u-1,v,n-2k-2)\\&+d_{2k-1_{a}}(u-1,v-1,n-4k-2),
\end{aligned}
\end{equation}
\begin{equation}
\label{eqd7}
\begin{aligned}
d_{2k+3_{a^2}}(u,v,n)=&\ d_{2k+2_{a}}(u,v,n)+d_{2k_{a}}(u-2,v,n-2k-3)\\&+d_{2k-1_{a}}(u-2,v-1,n-4k-3),
\end{aligned}
\end{equation}
\begin{equation}
\label{eqd8}
d_{2k+2_{b}}(u,v,n)=d_{2k+3_{a^2}}(u,v,n)+d_{2k+1_{a}}(u,v-1,n-2k-2).
\end{equation}
\end{lemma}

\begin{proof}
We prove equations~\eqref{eqd1} and~\eqref{eqd5}. Equations~\eqref{eqd2},~\eqref{eqd3},~\eqref{eqd4} and~\eqref{eqd8} are proved in the same way as equation~\eqref{eqd1}, and equations~\eqref{eqd6} and~\eqref{eqd7} in the same way as equation~\eqref{eqd5}.

Let us start with \eqref{eqd1}. We divide the set of partitions enumerated by $d_{2k+1_{ab}}(u,v,n)$ into two sets, those with largest part less than $2k+1_{ab}$ and those with largest part equal to $2k+1_{ab}$. Thus
\begin{equation*}
d_{2k+1_{ab}}(u,v,n)=d_{2k_{b}}(u,v,n) + e_{2k+1_{ab}}(u,v,n).
\end{equation*}
Let us now consider a partition $\lambda_1+ \lambda_2 + \cdots + \lambda_s$ counted by $ e_{2k+1_{ab}}(u,v,n).$ By the difference conditions given by the matrix $A$, $\lambda_2 \leq 2k-1_a$. Let us remove the largest part $\lambda_1=2k+1_{ab}$. The largest part is now $\lambda_2 \leq 2k-1_a$, the number partitioned is $n-(2k+1)$, and we removed a part coloured $ab$ so $u$ becomes $u-1$ and $v$ becomes $v-1$. We obtain a partition counted by $d_{2k-1_{a}}(u-1,v-1,n-2k-1)$. This process is reversible, so we have a bijection between partitions counted by $e_{2k+1_{ab}}(u,v,n)$ and those counted by $d_{2k-1_{a}}(u-1,v-1,n-2k-1)$.
Therefore $$e_{2k+1_{ab}}(u,v,n)=d_{2k-1_{a}}(u-1,v-1,n-2k-1)$$ for all $u,v \in \N$, $k,n \in \N^*$ and~\eqref{eqd1} is proved.

Let us now prove~\eqref{eqd5}.
Again let us divide the set of partitions enumerated by $d_{2k+2_{ab}}(u,v,n)$ into two sets, those with largest part less than $2k+2_{ab}$ and those with largest part equal to $2k+2_{ab}$. Thus
\begin{equation*}
d_{2k+2_{ab}}(u,v,n)=d_{2k+1_{b}}(u,v,n)+e_{2k+2_{ab}}(u,v,n).
\end{equation*}
Let us now consider a partition $\lambda_1+ \lambda_2 + \cdots + \lambda_s$ counted by $e_{2k+2_{ab}}(u,v,n)$. By the difference conditions of the matrix $A$, $\lambda_2=2k_b$ or $\lambda_2 \leq 2k_a$. Let us remove the largest part $\lambda_1=2k+2_{ab}$. If $\lambda_2=2k_b$, we obtain a partition counted by $e_{2k_b}(u-1,v-1,n-(2k+2))$. If $\lambda_2 \leq 2k_a$, we obtain a partition counted by $d_{2k_a}(u-1,v-1,n-(2k+2)).$ This process is also reversible and the following holds:
\begin{equation*}
e_{2k+2_{ab}}(u,v,n)= e_{2k_b}(u-1,v-1,n-(2k+2)) + d_{2k_a}(u-1,v-1,n-(2k+2)).
\end{equation*}
Moreover, again by removing the largest part, we can prove that $$e_{2k_b}(u-1,v-1,n-(2k+2))=d_{2k-1_a}(u-1,v-2,n-4k-2).$$
This concludes the proof of~\eqref{eqd5}.
\end{proof}

The equations of Lemma~\ref{eqd} lead to the following $q$-difference equations:

\begin{lemma}
For all $k \in \N^*,$
\label{equations}
\begin{equation}
\label{eq1}
G_{2k+1_{ab}}(a,b,q)=G_{2k_{b}}(a,b,q)+abq^{2k+1} G_{2k-1_{a}}(a,b,q),
\end{equation}
\begin{equation}
\label{eq2}
G_{2k+1_a}(a,b,q)=G_{2k+1_{ab}}(a,b,q)+aq^{2k+1} G_{2k_{ab}}(a,b,q),
\end{equation}
\begin{equation}
\label{eq3}
G_{2k+1_{b^2}}(a,b,q)=G_{2k+1_{a}}(a,b,q)+b^2q^{2k+1}G_{2k-1_{a}}(a,b,q),
\end{equation}
\begin{equation}
\label{eq4}
G_{2k+1_{b}}(a,b,q)=G_{2k+1_{b^2}}(a,b,q)+bq^{2k+1}G_{2k_{a}}(a,b,q),
\end{equation}
\begin{equation}
\label{eq5}
G_{2k+2_{ab}}(a,b,q)= G_{2k+1_{b}}(a,b,q)+abq^{2k+2}G_{2k_{a}}(a,b,q)+ab^2q^{4k+2}G_{2k-1_{a}}(a,b,q),
\end{equation}
\begin{equation}
\label{eq6}
G_{2k+2_{a}}(a,b,q)= G_{2k+2_{ab}}(a,b,q)+ aq^{2k+2} G_{2k_{a}}(a,b,q)+ abq^{4k+2}G_{2k-1_{a}}(a,b,q),
\end{equation}
\begin{equation}
\label{eq7}
G_{2k+3_{a^2}}(a,b,q)=G_{2k+2_{a}}(a,b,q)+a^2q^{2k+3}G_{2k_{a}}(a,b,q)+a^2bq^{4k+3}G_{2k-1_{a}}(a,b,q),
\end{equation}
\begin{equation}
\label{eq8}
G_{2k+2_{b}}(a,b,q)=G_{2k+3_{a^2}}(a,b,q)+bq^{2k+2}G_{2k+1_{a}}(a,b,q).
\end{equation}
\end{lemma}

These $q$-difference equations completely characterise the difference conditions and allow one to prove directly three cases of Proposition~\ref{keyprop}. We will introduce some more in the proof of the fourth case for convenience.

\section{Proof of Proposition~\ref{keyprop}}
\label{induction}
In this section we prove Proposition~\ref{keyprop} by induction. The proof resembles the one of~\cite{Doussesil}, but here we consider the non-dilated coloured integers and we keep track of the letters $a$ and $b$.

\subsection{Initialisation}
First we check some initial cases.

With the initial conditions
\begin{equation*}
G_{1_{ab}}(a,b,q)=1,
\end{equation*}
\begin{equation*}
G_{1_a}(a,b,q)=1 +aq,
\end{equation*}
\begin{equation*}
G_{1_{b^2}}(a,b,q)=1 +aq,
\end{equation*}
\begin{equation*}
G_{1_b}(a,b,q)=1 +aq+bq,
\end{equation*}
\begin{equation*}
G_{2_{ab}}(a,b,q)=1 +aq+bq+abq^2,
\end{equation*}
\begin{equation*}
G_{2_{a}}(a,b,q)=1 +aq+bq+abq^2+aq^2,
\end{equation*}
\begin{equation*}
G_{3_{a^2}}(a,b,q)=1 +aq+bq+abq^2+aq^2+a^2q^3,
\end{equation*}
\begin{equation*}
G_{2_{b}}(a,b,q)=1 +aq+bq+abq^2+aq^2+a^2q^3+bq^2+abq^3,
\end{equation*}
and equations~\eqref{eq1}-\eqref{eq8}, we use Maple to check that Proposition~\ref{keyprop} is verified for $k=1,2,3,4.$

Let us now assume that Proposition~\ref{keyprop} is true for all $\ell \leq k-1$ and show that it is also satisfied for $k$. To do so, we will prove the $4$ different equations~\eqref{key1}--\eqref{key4}.

\subsection{Equation~\eqref{key1}}
We start by proving that $$G_{2k+1_{ab}}(a,b,q) = (1+aq) G_{2k_{a}}(b,aq,q).$$

Replacing $k$ by $k-1$ in~\eqref{eq8} and substituting into~\eqref{eq1}, we obtain
\begin{equation}
\label{plic}
G_{2k+1_{ab}}(a,b,q)=G_{2k+1_{a^2}}(a,b,q)+\left(bq^{2k}+abq^{2k+1}\right) G_{2k-1_{a}}(a,b,q).
\end{equation}
We now replace $k$ by $k-1$ in~\eqref{eq2} and substitute into~\eqref{plic}. This gives
\begin{align*}
G_{2k+1_{ab}}(a,b,q)=& \ G_{2k+1_{a^2}}(a,b,q)+\left(bq^{2k}+abq^{2k+1}\right) G_{2k-1_{ab}}(a,b,q) \\
&+\left(abq^{4k-1}+a^2bq^{4k}\right)G_{2k-2_{ab}}(a,b,q).
\end{align*}
Then by the induction hypothesis,
\begin{equation}
\label{ploc}
\begin{aligned}
G_{2k+1_{ab}}(a,b,q)=(1+aq)\Big(&G_{2k-1_b}(b,aq,q)+\left(bq^{2k}+abq^{2k+1}\right)G_{2k-2_a}(b,aq,q)
\\&+\left(abq^{4k-1}+a^2bq^{4k}\right)G_{2k-3_a}(b,aq,q)\Big).
\end{aligned}
\end{equation}
Replacing $k$ by $k-1$ in~\eqref{eq5}, we obtain
\begin{equation}
\label{eq1*}
G_{2k_{ab}}(a,b,q)= G_{2k-1_{b}}(a,b,q)+abq^{2k}G_{2k-2_{a}}(a,b,q)+ab^2q^{4k-2}G_{2k-3_{a}}(a,b,q).
\end{equation}
Replacing $k$ by $k-1$ in~\eqref{eq6} gives
\begin{equation}
\label{eq2*}
G_{2k_{a}}(a,b,q)= G_{2k_{ab}}(a,b,q)+ aq^{2k} G_{2k-2_{a}}(a,b,q)+ abq^{4k-2}G_{2k-3_{a}}(a,b,q).
\end{equation}
Adding~\eqref{eq1*} and~\eqref{eq2*} and replacing $a$ by $b$ and $b$ by $aq$, we get
\begin{equation*}
\begin{aligned}
G_{2k_a}(b,aq,q)&=G_{2k-1_{b}}(b,aq,q)+ \left(bq^{2k} +abq^{2k+1}\right) G_{2k-2_{a}}(b,aq,q)
\\&+\left(abq^{4k-1} + a^2bq^{4k}\right)G_{2k-3_{a}}(b,aq,q).
\end{aligned}
\end{equation*}
Thus by~\eqref{ploc},we deduce that
\begin{equation*}
G_{2k+1_{ab}}(a,b,q)=(1+aq)G_{2k_a}(b,aq,q).
\end{equation*}

It remains now to prove that Equations~\eqref{key2}, \eqref{key3} and~\eqref{key4} are also satisfied.

\subsection{Equation~\eqref{key2}}
We now treat the second equation of Proposition~\ref{keyprop} and prove that
$$G_{2k+1_{b^2}}(a,b,q) = (1+aq) G_{2k_{b}}(b,aq,q).$$

By adding~\eqref{eq2} and~\eqref{eq3} together, we obtain
\begin{equation}
\label{cas2eq1}
G_{2k+1_{b^2}}(a,b,q)=G_{2k+1_{ab}}(a,b,q)+aq^{2k+1}G_{2k_{ab}}(a,b,q) +b^2q^{2k+1}G_{{2k-1}_{a}}(a,b,q).
\end{equation}
Replacing $k$ by $k-1$ in~\eqref{eq2} and substituting into~\eqref{cas2eq1}, we get
\begin{equation}
\label{cas2eq2}
\begin{aligned}
G_{2k+1_{b^2}}(a,b,q)=&\ G_{2k+1_{ab}}(a,b,q)+aq^{2k+1}G_{2k_{ab}}(a,b,q) \\
&+b^2q^{2k+1}G_{2k-1_{ab}}(a,b,q) +ab^2q^{4k} G_{2k-2_{ab}}(a,b,q).
\end{aligned}
\end{equation}
By the induction hypothesis,
\begin{equation}
\label{cas2eq3}
\begin{aligned}
G_{2k+1_{b^2}}(a,b,q)=&\ (1+aq) \Big[ G_{2k_{a}}(b,aq,q)+aq^{2k+1}G_{2k-1_{a}}(b,aq,q) \\
&+b^2q^{2k+1}G_{2k-2_{a}}(b,aq,q) +ab^2q^{4k} G_{2k-3_{a}}(b,aq,q) \Big].
\end{aligned}
\end{equation}
So we only need to prove that the expression between square brackets equals $G_{2k_b}(b,aq,q)$.

Adding~\eqref{eq7} and~\eqref{eq8} and replacing $k$ by $k-1$ gives
\begin{align*}
G_{2k_{b}}(a,b,q)=& \ G_{2k_{a}}(a,b,q)+bq^{2k}G_{2k-1_{a}}(a,b,q)
\\&+a^2q^{2k+1}G_{2k-2_{a}}(a,b,q)+a^2bq^{4k-1}G_{2k-3_{a}}(a,b,q).
\end{align*}
Replacing $a$ by $b$ and $b$ by $aq$ leads to
\begin{align*}
G_{2k_{b}}(b,aq,q)=& \ G_{2k_{a}}(b,aq,q)+aq^{2k+1}G_{2k-1_{a}}(b,aq,q)
\\&+b^2q^{2k+1}G_{2k-2_{a}}(b,aq,q)+b^2aq^{4k}G_{2k-3_{a}}(b,aq,q).
\end{align*}
Thus by~\eqref{cas2eq3}, $G_{2k+1_{b^2}}(a,b,q) = (1+aq) G_{2k_{b}}(b,aq,q)$, and~\eqref{key2} is proved.

\subsection{Equation~\eqref{key3}}
Let us now turn to Equation~\eqref{key3} and prove that
$$G_{2k+2_{ab}}(a,b,q) = (1+aq) G_{2k+1_{a}}(b,aq,q).$$

Substituting~\eqref{eq1} into~\eqref{eq2}, we have
\begin{equation}
\label{plouf}
G_{2k+1_a}(a,b,q)=G_{2k_{b}}(a,b,q)+aq^{2k+1} G_{2k_{ab}}(a,b,q)+abq^{2k+1} G_{2k-1_{a}}(a,b,q).
\end{equation}
Replacing $k$ by $k-1$ in~\eqref{eq5} and substituting in~\eqref{plouf}, we obtain
\begin{equation}
\label{plouf3}
\begin{aligned}
G_{2k+1_a}(a,b,q)=&\ G_{2k_{b}}(a,b,q)+aq^{2k+1} G_{2k-1_{b}}(a,b,q) + a^2bq^{4k+1} G_{2k-2_{a}}(a,b,q)
\\&+a^2b^2q^{6k-1} G_{2k-3_{a}}(a,b,q)+abq^{2k+1} G_{2k-1_{a}}(a,b,q).
\end{aligned}
\end{equation}
Then replacing $a$ by $b$ and $b$ by $aq$ in~\eqref{plouf3}, we obtain the following equation:
\begin{equation}
\label{etoile}
\begin{aligned}
G_{2k+1_a}(b,aq,q)=& \ G_{2k_{b}}(b,aq,q)+bq^{2k+1} G_{2k-1_{b}}(b,aq,q) \\
&+ ab^2q^{4k+2} G_{2k-2_{a}}(b,aq,q)+a^2b^2q^{6k+1} G_{2k-3_{a}}(b,aq,q)\\
&+abq^{2k+2} G_{2k-1_{a}}(b,ab,q).
\end{aligned}
\end{equation}
Thus we want to prove that
\begin{equation}
\label{paf}
\begin{aligned}
G_{2k+2_{ab}}(a,b,q)=&\ G_{2k+1_{b^2}}(a,b,q)+bq^{2k+1} G_{2k+1_{a^2}}(a,b,q)\\
& + ab^2q^{4k+2} G_{2k-1_{ab}}(a,b,q)+a^2b^2q^{6k+1} G_{2k-2_{ab}}(a,b,q)\\
&+abq^{2k+2} G_{2k_{ab}}(a,b,q).
\end{aligned}
\end{equation}

We add Equations~\eqref{eq4} and~\eqref{eq5} together and get
\begin{equation}
\label{star}
\begin{aligned}
G_{2k+2_{ab}}(a,b,q)=&\ G_{2k+1_{b^2}}(a,b,q)+bq^{2k+1} G_{2k_{a}}(a,b,q) + ab^2q^{4k+2} G_{2k-1_{a}}(a,b,q)
\\&+abq^{2k+2} G_{2k_{a}}(a,b,q).
\end{aligned}
\end{equation}
Replacing $k$ by $k-1$ in~\eqref{eq2},~\eqref{eq6} and~\eqref{eq7} and substituting into~\eqref{star} yields
\begin{align*}
G&_{2k+2_{ab}}(a,b,q)=\ G_{2k+1_{b^2}}(a,b,q)\\
&+bq^{2k+1} \left(G_{2k+1_{a^2}}(a,b,q)- a^2q^{2k+1} G_{2k-2_a}(a,b,q) -a^2bq^{4k-1} G_{2k-3_a}(a,b,q)\right) \\
&+ ab^2q^{4k+2} \left(G_{2k-1_{ab}}(a,b,q) +aq^{2k-1} G_{2k-2_{ab}}(a,b,q) \right)\\
&+abq^{2k+2} \left(G_{2k_{ab}}(a,b,q) + aq^{2k}G_{2k-2_a}(a,b,q) +abq^{4k-2}G_{2k-3_a}(a,b,q) \right).
\end{align*}
The two terms with a minus simplify with the last two terms and we obtain~\eqref{paf}.
Thus $G_{2k+2_{ab}}(a,b,q) = (1+aq) G_{2k+1_{a}}(b,aq,q)$ and~\eqref{key3} is proved. Let us now turn to the last and most difficult equation of Proposition~\ref{keyprop}.

\subsection{Equation~\eqref{key4}}

Finally, it remains to prove that
$$G_{2k+1_{a^2}}(a,b,q) = (1+aq) G_{2k-1_{b}}(b,aq,q).$$

Adding~\eqref{eq3} and~\eqref{eq4} together and replacing $a$ by $b$ and $b$ by $aq$ leads to

\begin{equation}
\label{eq34}
G_{2k+1_b}(b,aq,q)=G_{2k+1_a}(b,aq,q)+aq^{2k+2}G_{2k_a}(b,aq,q)+a^{2}q^{2k+3}G_{2k-1_a}(b,aq,q).
\end{equation}
We now want to show that
\begin{equation}
\label{goal}
G_{2k+3_{a^2}}(a,b,q)=G_{2k+2_{ab}}(a,b,q)+aq^{2k+2}G_{2k+1_{ab}}(a,b,q)+a^{2}q^{2k+3}G_{2k_{ab}}(a,b,q).
\end{equation}
To do so, we introduce some new recurrences. Recurrences~\eqref{eqd1}--\eqref{eqd8} would theoretically have been sufficient, as they completely characterise the partitions we consider, but the proof would imply too many substitutions and be very long. By definition we have 
\begin{equation}
\label{pif3}
d_{2k+3_{a^2}}(u,v,n)=d_{2k+2_{ab}}(u,v,n)+e_{2k+2_a}(u,v,n)+e_{2k+3_{a^2}}(u,v,n).
\end{equation}
In a similar manner as above, by the difference conditions of the matrix $A$, and removing the largest part, we show that 
\begin{equation}
\label{pif1}
\begin{aligned}
e_{2k+2_a}(u,v,n)&=d_{2k+1_{ab}}(u-1,v,n-(2k+2))
\\&-e_{2k+1_{a^2}}(u-1,v,n-(2k+2))-e_{2k+1_{ab}}(u-1,v,n-(2k+2)),
\end{aligned}
\end{equation}
and 
\begin{equation}
\label{pif2}
\begin{aligned}
e_{2k+3_{a^2}}(u,v,n)&=d_{2k_{ab}}(u-2,v,n-(2k+3))
\\&+e_{2k_{a}}(u-2,v,n-(2k+3))+e_{2k_{b}}(u-2,v,n-(2k+3)),
\end{aligned}
\end{equation}
By removing the largest part again, we show that
\begin{equation*}
e_{2k+1_{ab}}(u-1,v,n-(2k+2))=d_{2k-1_a}(u-2,v-1,n-(4k+3)),
\end{equation*}
and
\begin{equation*}
e_{2k_{b}}(u-2,v,n-(2k+3))=d_{2k-1_a}(u-2,v-1,n-(4k+3)).
\end{equation*}
Therefore
\begin{equation*}
e_{2k+1_{ab}}(u-1,v,n-(2k+2))=e_{2k_{b}}(u-2,v,n-(2k+3)).
\end{equation*}
And in the same way
\begin{align*}
e_{2k+1_{a^2}}(u-1,v,n-(2k+2))=&\ e_{2k-2_{b}}(u-3,v,n-(4k+3))\\
&+d_{2k-2_{a}}(u-3,v,n-(4k+3)),
\end{align*}
and
\begin{equation*}
e_{2k_{a}}(u-2,v,n-(2k+2))=e_{2k-2_{b}}(u-3,v,n-(4k+3))+d_{2k-2_{a}}(u-3,v,n-(4k+3)).
\end{equation*}
Therefore
\begin{equation*}
e_{2k+1_{a^2}}(u-1,v,n-(2k+2))=e_{2k_{a}}(u-2,v,n-(2k+2)).
\end{equation*}
So by plugging~\eqref{pif1} and~\eqref{pif2} into~\eqref{pif3}, we get
\begin{align*}
d_{2k+3_{a^2}}(u,v,n)=&\ d_{2k+2_{ab}}(u,v,n)
\\&+d_{2k+1_{ab}}(u-1,v,n-(2k+2))+d_{2k_{ab}}(u-2,v,n-(2k+3)),
\end{align*}
which gives in terms of generating functions
\begin{equation*}
G_{2k+3_{a^2}}(a,b,q)=G_{2k+2_{ab}}(a,b,q)+aq^{2k+2}G_{2k+1_{ab}}(a,b,q)+a^2q^{2k+3}G_{2k_{ab}}(a,b,q).
\end{equation*}
This is exactly~\eqref{goal}, so by the induction hypothesis and the results from the last two subsections,
$$G_{2k+1_{a^2}}(a,b,q) = (1+aq) G_{2k-1_{b}}(b,aq,q).$$
This concludes the proof of Proposition~\ref{keyprop}.

\section{Conclusion}
The proof of this paper shows that the method of weighted words can be combined with $q$-difference equations to prove theorems with intricate difference conditions, where it is hard to compute the minimal partition as in the classical method of weighted words. It also gives a first example of the method of weighted words where some squared colours appear. It would be interesting to see if this method can be applied to other identities, such as Andrews's theorems~\cite{Generalisation2,Generalisation1} or its generalisations to overpartitions~\cite{Doussegene2,Doussegene}, or other identities arising from the theory of vertex operators or Lie algebras like those of Primc~\cite{Primc} and Meurman-Primc~\cite{Meurman} for example. Introducing squared colours in the method of weighted words for G\"ollinitz' theorem~\cite{AllAndGor2} might also lead to some interesting new identities.
Finally, as this new refinement gives more combinatorial insight on Siladi\'c's identity, it would be interesting to know if a bijective proof can be found.

\section*{Acknowledgements}
The author thanks Jeremy Lovejoy for very helpful discussions.

\bibliographystyle{siam}
\bibliography{biblio}

\end{document}